%% file: muncenteredcharf.tex
\documentclass{article}
\input{packages}
\input{commands}

\begin{document}

\title{Variation of the uncentered maximal characteristic Function}
\author{Julian Weigt}
\affil{Department of Mathematics and Systems Analysis, Aalto University, Finland, \texttt{julian.weigt@aalto.fi}}
\maketitle

\begin{abstract}
\input{abstract.tex}
\end{abstract}

\begingroup
\renewcommand\thefootnote{}\footnotetext{%
2020 \textit{Mathematics Subject Classification.} 42B25,26B30.\\%
\textit{Key words and phrases.} Maximal function, variation, dyadic cubes.%
}%
\addtocounter{footnote}{-1}%
\endgroup

\section*{Introduction}

\input{introduction.tex}

\section{Preliminaries and main result}\label{sec_preliminaries}

\input{preliminaries.tex}

\section{Tools for both maximal operators}\label{sec_both}

\input{general.tex}

\section{The dyadic maximal function}\label{sec_dyadic}

\input{characteristicf_levelsets.tex}

\section{The uncentered maximal function}\label{sec_uncentered}

\input{uncentered.tex}

\section{The optimal rate in \texorpdfstring{\(\lambda\)}{lambda}}\label{sec_optimal}

\input{optimal.tex}

\nocite{*}
\bibliographystyle{plain}
\bibliography{bib}

\end{document}

%% file: packages.tex
\usepackage{amsmath,amsfonts,amssymb,amsthm,mathtools,enumitem,geometry,esint,color,hyperref,cleveref,authblk}

\setlist[enumerate]{label = (\roman*)}
\crefname{enumi}{}{}

\newlist{caseenum}{enumerate}{1}
\setlist[caseenum]{label = (\roman*)}
\crefname{caseenumi}{case}{cases}

\crefname{equation}{}{}
\Crefname{equation}{Equation}{Equations}

\theoremstyle{plain}
\newtheorem{theo}{Theorem}[section]

\theoremstyle{definition}

\newtheorem*{exa*}{Example}

\newtheorem{lem}[theo]{Lemma}
\newtheorem*{lem*}{Lemma}

\newtheorem{pro}[theo]{Proposition}
\newtheorem*{pro*}{Proposition}

\newtheorem*{cla*}{Claim}

\newtheorem*{lie*}{Lie}

\theoremstyle{remark}
\newtheorem{rem}[theo]{Remark}
\newtheorem*{rem*}{Remark}

\newtheorem*{note*}{Note}

\crefname{cor}{corollary}{corollaries}
\crefname{pro}{proposition}{propositions}
\crefname{lem}{lemma}{lemmas}
\crefname{theo}{theorem}{theorems}

\allowdisplaybreaks

%% file: commands.tex
\newcommand\intd{\mathop{}\!\mathrm{d}}
\newcommand\df\emph
\newcommand\ind[1]{1_{#1}}
\newcommand\tx\text
\newcommand\f\frac

\newcommand\loc{\mathrm{loc}}

\newcommand\Q{\mathcal Q}
\newcommand\B{\mathcal B}
\newcommand\C{\mathcal C}
\newcommand\F{\mathcal F}
\newcommand\G{\mathcal G}

\newcommand\avint\fint

\DeclareMathOperator\var{var}
\DeclareMathOperator\diam{diam}
\DeclareMathOperator\dist{dist}

\let\div\relax
\DeclareMathOperator\div{div}

\newcommand\M{{\mathrm M}}

\newcommand\ub\underbrace

\newcommand\la\langle
\newcommand\ra\rangle

\newcommand\comp\complement
\newcommand\mb[1]{\partial_*{#1}}
\newcommand\rb[1]{\partial^*{#1}}
\newcommand\mc[1]{\overline{#1}^*}
\newcommand\mib[1]{\mathrm{int}_*{\Bigl(#1\Bigr)}}
\newcommand\mi[1]{\mathrm{int}_*{(#1)}}
\newcommand\cl[1]{\overline{#1}}
\newcommand\ti[1]{\mathring{#1}}
\newcommand\sm[1]{\mathcal{H}^{d-1}(#1)}
\newcommand\smb[1]{\mathcal{H}^{d-1}\Bigl(#1\Bigr)}
\newcommand\lm[1]{\mathcal{L}(#1)}
\newcommand\lmb[1]{\mathcal{L}\Bigl(#1\Bigr)}

\newcommand\proj[1]{\widehat{#1}}

%% file: abstract.tex
Let \(\M\) be the uncentered Hardy-Littlewood maximal operator or the dyadic maximal operator and \(d\geq1\).
We prove that for a set \(E\subset\mathbb{R}^d\) of finite perimeter
the bound \(\var\M\ind E\leq C_d\var\ind E\) holds. 
We also prove this for the local maximal operator.

%% file: introduction.tex
The uncentered Hardy-Littlewood maximal function of a non-negative locally integrable function \(f\) is given by
\[\M f(x)=\sup_{B\ni x}\f1{\lm{B}}\int_B f,\]
where the supremum is taken over all open balls \(B\subset\mathbb{R}^d\) that contain \(x\).
Various versions of this maximal operator have been investigated. 
There is the (centered) Hardy-Littlewood maximal operator, where the supremum is taken only over those balls that are centered in \(x\), 
or the dyadic maximal operator which maximizes over dyadic cubes instead of balls.
Those operators also have local versions, where for some open set \(\Omega\subset\mathbb{R}^d\) the supremum is taken only over those balls or cubes that are contained in \(\Omega\).
For example the local dyadic maximal function with respect to \(\Omega\) of \(f\in L^1_\loc(\Omega)\) at \(x\in\Omega\) is given by
\[\M f(x)=\sup_{x\in Q\subset \Omega}\f1{\lm{Q}}\int_Q f,\]
where the supremum is taken over all half open dyadic cubes \(Q\subset\mathbb{R}^d\) with \(x\in Q\subset\Omega\).

It is well known that many maximal operators are bounded on \(L^p(\mathbb{R}^d)\) if and only if \(p>1\). 
The regularity of the maximal operator was first studied in \cite{MR1469106}, 
where Kinnunen proved for the Hardy-Littlewood maximal operator that for \(p>1\) and \(f\in W^{1,p}(\mathbb{R}^d)\) also the bound
\[\|\nabla\M f\|_p\leq C_{d,p}\|\nabla f\|_p\]
holds, from which it follows that the Hardy-Littlewood maximal operator is bounded on \(W^{1,p}(\mathbb{R}^d)\). 
The proof combines the point-wise bound \(|\nabla\M f|\leq\M|\nabla f|\) with the \(L^p(\mathbb{R}^d)\)-bound of the maximal operator.
Since the maximal operator is not bounded on \(L^1(\mathbb{R}^d)\), this approach fails for \(p=1\).
For \(p>1\) the gradient \(L^p(\mathbb{R}^d)\)-bound or some corresponding version is valid for most maximal operators.
However so far no counterexamples have been found for \(p=1\). 
So in 2004, Haj\l{}asz and Onninen posed the following question in \cite{MR2041705}:
For the Hardy-Littlewood maximal operator \(\M\), is \(f\mapsto|\nabla\M f|\) a bounded mapping \(W^{1,1}(\mathbb{R}^d)\rightarrow L^1(\mathbb{R}^d)\)? 
This question for various maximal operators has since become a well known problem
and has been the subject of lots of research.
In one dimension for \(L^1(\mathbb{R})\) the gradient bound has already been proven in \cite{MR1898539} by Tanaka for the uncentered maximal function,
and later in \cite{MR3310075} by Kurka for the centered Hardy-Littlewood maximal function. 
The latter proof turned out to be much more complicated. 
In \cite{MR3800463} Luiro has proven the gradient bound for radial functions in \(L^1(\mathbb{R}^d)\) for the uncentered maximal operator.
More research on this question, 
and also more generally on the endpoint regularity of maximal operators 
can be found in
\cite{MR2868961,MR2276629,MR2356061,MR3091605,MR3624402,MR3695894,MR2550181,MR3592548}.
However, so far the question has been essentially unsolved in dimensions larger than one for any maximal operator.

In this paper we prove that for \(\M\) being the dyadic or the uncentered Hardy Littlewood maximal operator
and \(E\subset\mathbb{R}^d\) being a set with finite perimeter, we have
\[\var\M\ind E\leq C_d\var\ind E.\]
This answers the question of Haj\l{}asz and Onninen in a special case,
and is the first truly higher dimensional result for \(p=1\) to the best of our knowledge.
We furthermore prove a localized version, as is stated in \Cref{theo_goaldyadic,theo_goal}.
The Hardy-Littlewood uncentered maximal function and the dyadic maximal function have in common
that their levels sets \(\{\M f>\lambda\}\) can be written as
the union of all balls/dyadic cubes \(X\) with \(\int_X f>\lambda\lm X\).
Our proof relies on this.
Since this is not true for the centered Hardy-Littlewood maximal function,
a different approach has to be found for that maximal operator.

Also related topics for various exponents \(1\leq p\leq\infty\) have been studied,
such as the continuity of the maximal operator in Sobolev spaces \cite{MR1724375}
and bounds for the gradient of other maximal operators, such as fractional, convolution, discrete, local and bilinear maximal operators
\cite{MR3809456,MR2431055,MR3063097,MR3319617,MR1650343,MR1979008,2017arXiv171007233L,2019arXiv190904375R}.

I would like to thank my supervisor, Juha Kinnunen for all of his support,
Panu Lahti for discussions on the theory of sets of finite perimeter,
his suggested proof of \Cref{cla_largeboundaryinball},
and repeated reading of and advice on the manuscript,
and Carlos Mudarra for discussions on the proof of \Cref{cla_surfacedistanceboxingball}.
I am indebted to the anonymous referees for their careful reading and the large amount of improvements they suggested both in style and in mathematical substance to the paper.
The author has been supported by the Vilho, Yrj\"o and Kalle V\"ais\"al\"a Foundation of the Finnish Academy of Science and Letters.

%% file: preliminaries.tex
We work in the setting of sets of finite perimeter, as in Evans-Gariepy \cite{MR3409135}, Section~5. 
For a measurable set \(E\subset\mathbb{R}^d\) we denote by \(\lm E\) its Lebesgue measure
and by \(\sm E\) its \(d-1\)-dimensional Hausdorff measure.
For an open set \(\Omega\subset\mathbb{R}^d\), a function \(f\in L^1_\loc(\Omega)\) is said to have locally bounded variation
if for each open and compactly supported \(U\subset \Omega\) we have
\[\sup\Bigl\{\int_Uf\div\varphi:\varphi\in C^1_{\tx c}(U;\mathbb{R}^d),\ |\varphi|\leq1\Bigr\}<\infty.\]
Such a function comes with a measure \(\mu\) and a function \(\nu:\Omega\rightarrow\mathbb{R}^d\) that has \(|\nu|=1\) \(\mu\)-a.e.\ such that for all \(\varphi\in C^1_{\tx c}(\Omega;\mathbb{R}^d)\) we have
\[\int_\Omega f\div\varphi=\int_\Omega\varphi\nu\intd \mu.\]
We define the variation of \(f\) in \(\Omega\) by
\[\var_\Omega f=\mu(\Omega).\]
For a measurable set \(E\subset\mathbb{R}^d\) we define the measure theoretic boundary by
\[\mb E=\Bigl\{x:\limsup_{r\rightarrow0}\f{\lm{B(x,r)\setminus E}}{r^d}>0,\ \limsup_{r\rightarrow0}\f{\lm{B(x,r)\cap E}}{r^d}>0\Bigr\}.\]
The following coarea formula is our strategy to approach the variation of the maximal function.

\begin{lem}[Theorem~5.9 in \cite{MR3409135}]\label{lem_coareabv}
Let \(\Omega\subset\mathbb{R}^d\) be open. 
Let \(f\in L^1_\loc(\Omega)\). 
Then
\[\var_\Omega f=\int_\mathbb{R}\sm{\mb{\{f>\lambda\}\cap \Omega}}\intd\lambda.\]
\end{lem}

We say that measurable set \(E\subset\mathbb{R}^d\) has locally finite perimeter
if its characteristic function \(\ind E\) has locally bounded variation.
For \(f=\ind E\) we call \(\var_\Omega\ind E\) the perimeter of \(E\)
and \(\nu\) from above the outer normal of \(E\).
\Cref{lem_coareabv} implies
\[\var_\Omega\ind E=\sm{\mb E\cap \Omega}.\]
Recall the definition of the set of dyadic cubes
\[\bigcup_{n\in\mathbb{Z}}\{[x_1,x_1+2^n)\times\ldots\times[x_d,x_d+2^n):i=1,\ldots,n,\ x_i\in2^n\mathbb{Z}\}.\]
The maximal function of a characteristic function can be written as
\[\M \ind E(x)=\sup_{x\in X\subset \Omega}\f{\lm{E\cap X}}{\lm{X}},\]
where \(X\) ranges over balls for the uncentered maximal operator,
and over dyadic cubes for the dyadic maximal operator.
Now we are ready to state the main results of this paper.

\begin{theo}\label{theo_goaldyadic}
Let \(\M\) be the local dyadic maximal operator with respect to an open set \(\Omega\subset\mathbb{R}^d\).
Let \(E\subset\mathbb{R}^d\) be a set with locally finite perimeter. 
Then
\[\var_\Omega\M\ind E\leq C_d\sm{\mb E\cap \Omega},\]
where \(C_d\) depends only on the dimension \(d\).
\end{theo}

\begin{theo}\label{theo_goal}
Let \(\M\) be the local uncentered maximal operator with respect to an open set \(\Omega\subset\mathbb{R}^d\).
Let \(E\subset\mathbb{R}^d\) be a set with locally finite perimeter. 
Then
\[\var_\Omega\M\ind E\leq C_d\sm{\mb E\cap \Omega},\]
where \(C_d\) depends only on the dimension \(d\).
\end{theo}

We can for example take \(\Omega=\mathbb{R}^d\).
Denote \(\{\M\ind E>\lambda\}=\{x\in\Omega:\M\ind E(x)>\lambda\}\).
We reduce \Cref{theo_goaldyadic,theo_goal} to the following results.

\begin{pro}\label{eq_levelsetsdyadic}
Let \(\M\) be the local dyadic maximal operator with respect to some open set \(\Omega\subset\mathbb{R}^d\).
Let \(E\subset\mathbb{R}^d\) be a set with locally finite perimeter and let \(\lambda\in(0,1)\). 
Then 
\[\sm{\mb{\{\M\ind E> \lambda\}}\cap \Omega}\leq C_d\lambda^{-\f{d-1}d}\sm{\mb E\cap \Omega}.\]
\end{pro}
By \Cref{lem_finMf} we have \(\mc E\cap\Omega\subset\mc{\{\M\ind E>\lambda\}}\) so that we might intersect the right-hand side with \(\mc{\{\M\ind E>\lambda\}}\).

\begin{pro}\label{eq_levelsets}
Let \(\M\) be the local uncentered maximal operator. 
Let \(E\subset\mathbb{R}^d\) be a set with locally finite perimeter and let \(\lambda\in(0,1)\). 
Then 
\[\sm{\mb{\{\M\ind E> \lambda\}}\cap\Omega}\leq C_d\lambda^{-\f{d-1}d}(1-\log\lambda)\sm{\mb E\cap\{\M\ind E>\lambda\}}.\]
\end{pro}

The constants \(C_d\) that appear in \Cref{theo_goal,theo_goaldyadic} and \Cref{eq_levelsetsdyadic,eq_levelsets} are not equal.
Since the proofs of \Cref{theo_goaldyadic,theo_goal} are almost the same we do them simultaneously.

\begin{proof}[Proof of \Cref{theo_goaldyadic,theo_goal}]
By \Cref{lem_coareabv} and \Cref{eq_levelsetsdyadic,eq_levelsets} we have
\begin{align*}
\var_\Omega\M\ind E&=\int_0^1\sm{\mb{\{\M\ind E> \lambda\}}\cap \Omega}\intd\lambda\\
&\leq C_d\int_0^1\lambda^{-\f{d-1}d}(1-\log\lambda)\sm{\mb E\cap \Omega}\intd\lambda\\
&=d(d+1)C_d\sm{\mb E\cap \Omega}.
\end{align*}
\end{proof}

In \Cref{sec_both,sec_dyadic,sec_uncentered} we prove \Cref{eq_levelsetsdyadic,eq_levelsets}.
In \Cref{sec_optimal} we prove \Cref{eq_levelsets_optimal} which is \Cref{eq_levelsets} without the factor \(1-\log\lambda\).
The rate \(\lambda^{-\f{d-1}d}\) is optimal.

We introduce some notation we will use throughout the paper.
By \(a\lesssim b\) we mean that there exists a constant \(C_d\) that depends only on the dimension \(d\) such that \(a\leq C_d b\). 
For a set \(\B\) of subsets of \(\mathbb{R}^d\) we write
\[\bigcup\B=\bigcup_{B\in\B}B.\]
For a ball \(B=B(x,r)\subset\mathbb{R}^d\) and \(c>0\) we denote
\(cB=B(x,cr)\).
If \(\B\) is a set of balls we denote
\[c\B=\{cB:B\in\B\}.\]
For a set \(E\subset\mathbb{R}^d\) and a point \(x\in\mathbb{R}^d\) we denote
\[
\dist(x,E)
=
\inf_{y\in E}|x-y|
.
\]
We also need more measure theoretic quantities.
We define the measure theoretic interior by
\begin{align*}
\mi E&=\Bigl\{x:\limsup_{r\rightarrow0}\f{\lm{B(x,r)\setminus E}}{r^d}=0\Bigr\},\\
\intertext{the measure theoretic closure by}
\mc E&=\Bigl\{x:\limsup_{r\rightarrow0}\f{\lm{B(x,r)\cap E}}{r^d}>0\Bigr\}
\end{align*}
and the measure theoretic boundary by
\[
\mb E
=
\mc E
\setminus
\mi E
.
\]

\begin{lem}\label{lem_boundaryofunion}
Let \(A,B\subset\mathbb{R}^d\) be measurable. 
Then
\[\mb{(A\cup B)}\subset(\mb A\setminus\mc B)\cup(\mb B\setminus\mc A)\cup(\mb A\cap\mb B).\]
\end{lem}

\begin{proof} Let \(x\in\mb{(A\cup B)}\). 
Then
\begin{align*}
\limsup_{r\rightarrow0}\f{\lm{B(x,r)\cap(A\cup B)}}{r^d}&>0\\
\intertext{
and
}
\limsup_{r\rightarrow0}\f{\lm{B(x,r)\setminus(A\cup B)}}{r^d}&>0.
\end{align*}
By symmetry it suffices to consider the case that
\[\limsup_{r\rightarrow0}\f{\lm{B(x,r)\cap A}}{r^d}>0.\]
Then
\[\limsup_{r\rightarrow0}\f{\lm{B(x,r)\setminus A}}{r^d}\geq\limsup_{r\rightarrow0}\f{\lm{B(x,r)\setminus(A\cup B)}}{r^d}>0\]
which means \(x\in\mb A\). 
Analogously, if
\[\limsup_{r\rightarrow0}\f{\lm{B(x,r)\cap B}}{r^d}>0\]
then \(x\in\mb B\) so we get \(x\in\mb A\cap\mb B\). 
Otherwise
\[\limsup_{r\rightarrow0}\f{\lm{B(x,r)\cap B}}{r^d}=0\]
and we can conclude \(x\in\mb A\setminus\mc B\).
\end{proof}

Let \(E\subset\mathbb{R}^d\) be measurable and let \(\mu\) be the measure from the definition of \(\var\ind E\) and \(\nu\) the outer normal.
We define the reduced boundary \(\rb E\) of \(E\subset\mathbb{R}^d\) as the set of all points \(x\in\mathbb{R}^d\) such that
for all \(r>0\) we have \(\mu(B(x,r))>0\),
\[\lim_{r\rightarrow0}\avint_{B(x,r)}\nu\intd\mu=\nu(x),\]
and \(|\nu(x)|=1\).
This is Definition~5.4 in \cite{MR3409135}.
By Lemma~5.5 in \cite{MR3409135} we have \(\rb E\subset\mb E\) and \(\sm{\mb E\setminus\rb E}=0\).
Thus it suffices to consider only the reduced boundary when estimating the perimeter of a set.
But most of the time we will formulate the results for the measure theoretic boundary. 
The exception is \Cref{lem_finMf}, which we could only prove for the reduced boundary because there we make use of Theorem~5.13 in \cite{MR3409135}, which states the following.

\begin{lem}[Theorem~5.13 in \cite{MR3409135}]\label{lem_blowuprb}
Let \(E\subset\mathbb{R}^d\) be a measurable set.
Assume \(0\in\rb E\) with \(\nu(0)=(1,0,\ldots,0)\). 
Then for \(r\rightarrow0\) we have
\(\ind{\f1rE}\rightarrow\ind{\{x:x_1<0\}}\)
in \(L^1_\loc(\mathbb{R}^d)\).
\end{lem}

A central tool used here is the relative isoperimetric inequality, see Theorem~5.11 in \cite{MR3409135}.
It states that for a ball \(B\) and any measurable set \(E\subset\mathbb{R}^d\) we have
\begin{equation}\label{eq_isoperimetric}
\min\{\lm{E\cap B},\lm{B\setminus E}\}^{d-1}\lesssim\sm{\mb E\cap B}^d.
\end{equation}
However we need the relative isoperimetric inequality also for other sets than balls.
An open bounded set \(A\) is called a John domain if there is a constant \(K\) and point \(x\in A\) from which every other point \(y\in A\) can be reached via a path \(\gamma\) such that for all \(t\) we have
\begin{equation}\label{eq_twistedcone}
\dist(\gamma(t),A^\comp)\geq K^{-1}|y-\gamma(t)|.
\end{equation}
\begin{figure}
\centering
\includegraphics{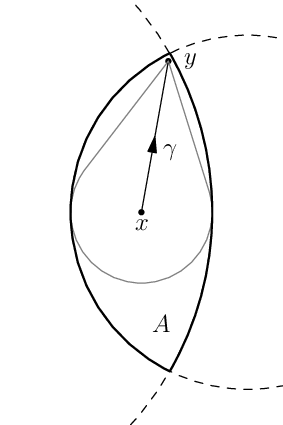}
\caption{A John domain.}
\label{fig_Acone}
\end{figure}
This is called the cone condition, see \Cref{fig_Acone}.
Theorem 107 in the lecture notes \cite{lecturenoteshajlasz} by Piotr Haj\l{}asz states that all John domains admit a relative isoperimetric inequality.

\begin{lem}\label{lem_isoperimetric}
Let \(A\subset\mathbb{R}^d\) be a John domain with constant \(K\). 
Then \(A\) satisfies a relative isoperimetric inequality with constant \(C_{K,d}\) only depending on \(K\) and the dimension \(d\),
\[\min\{\lm{E\cap A},\lm{A\setminus E}\}^{d-1}\leq C_{K,d}\sm{\mb E\cap A}^d.\]
\end{lem}
For example a ball and an open cube are John domains.


Another basic tool is the Vitali covering lemma, see for example Theorem~1.24 in \cite{MR3409135}.
\begin{lem}[Vitali covering lemma]\label{lem_disjointcover}
Let \(\B\) be a set of balls in \(\mathbb{R}^d\) with diameter bounded by some \(R\in\mathbb{R}\). 
Then it has a countable subset \(\tilde\B\) of disjoint balls such that
\[\bigcup\B\subset\bigcup5\tilde\B.\]
\end{lem}

Instead of considering \(\{\M\ind E>\lambda\}\) we will only consider a finite union of balls/cubes.
In order to pass from there to the whole set \(\{\M\ind E>\lambda\}\) we will use an approximation result.
We say that a sequence \((A_n)_n\) of sets in \(\mathbb{R}^d\) converges to some set \(A\) in \(L^1_\loc(\mathbb{R}^d)\) if \((\ind{A_n})_n\) converges to \(\ind A\) in \(L^1_\loc(\mathbb{R}^d)\). 

\begin{lem}[Theorem~5.2 in \cite{MR3409135} for characteristic functions]\label{lem_l1approx}
Let \(\Omega\subset\mathbb{R}^d\) be an open set and let \((A_n)_n\) be subsets of \(\mathbb{R}^d\) with locally finite perimeter that converge to \(A\) in \(L^1_\loc(\Omega)\). 
Then
\[\sm{\mb A\cap \Omega}\leq\liminf_{n\rightarrow\infty}\sm{\mb A_n\cap \Omega}.\]
\end{lem}

\begin{lem}
\label{eq_unitvolumeratio}
Let \(\sigma_d=\lm{B(0,1)}\) be the Lebesgue measure of the \(d\)-dimensional unit ball.
Then
\[
\sqrt{\f{2\pi}{d+1}}
\leq
\f{\sigma_d}{\sigma_{d-1}}
\leq
\sqrt{\f{2\pi}d}
.
\]
\end{lem}
\begin{proof}
By the logarithmic convexity of the \(\Gamma\)-function, for all \(x>\f12\) we have
\begin{align*}
\f{\Gamma(x)}{\Gamma(x+1/2)}
&\leq
\f{\sqrt{\Gamma(x-1/2)\Gamma(x+1/2)}}{\Gamma(x+1/2)}
=
\sqrt{\f{\Gamma(x-1/2)}{\Gamma(x+1/2)}}
=
\f1{\sqrt{x-1/2}}
,
\\
\f{\Gamma(x)}{\Gamma(x+1/2)}
&\geq
\f{\Gamma(x)}{\sqrt{\Gamma(x)\Gamma(x+1)}}
=
\sqrt{\f{\Gamma(x)}{\Gamma(x+1)}}
=
\f1{\sqrt x}
,
\end{align*}
and the result follows from
\(
\sigma_d
=\pi^{\f d2}/
\Gamma(d/2+1)
.
\)
\end{proof}

We will need some facts about convex sets.
\begin{lem}
\label{lem_surfaceconvex}
The following properties hold for all convex and bounded sets \(A,B\subset\mathbb{R}^d\).
\begin{enumerate}
\item
The set \(A\cap B\) is convex.
\label{it_intersectionconvex}
\item
If \(A\subset B\) then
\(
\sm{\partial A}
\leq
\sm{\partial B}
.
\)
\label{it_perimeterconvex}
\item
For every \(\varepsilon>0\) we have
\(
\lm{\{x\in A:0<\dist(x,A^\comp)\leq\varepsilon\}}
\leq\varepsilon
\sm{\partial A}
.
\)
\label{it_perimeterblowupcovex}
\end{enumerate}
\end{lem}
\begin{proof}
\Cref{it_intersectionconvex} follows from the definition of convexity.

For every \(x\in\partial B\) there is a point \(z\in\partial A\) with
\[
|z-x|
=
\min_{y\in\partial A}|y-x|
.
\]
A straightforward computation shows that if \(z'\in\partial A\) with \(|x-z'|=\min_{y\in\partial A}|x-y|\)
then
\(
|x-(z+z')/2|
\leq
\min_{y\in\partial A}|x-y|
\)
and the inequality is strict if \(z'\neq z\).
Hence we must have \(z'=z\) because \((z+z')/2\in\cl A\) by convexity.
We denote \(p(x)=z\).

Since \(A\) is convex, in every point \(z\in\partial A\) there is a hyperplane \(H\) which contains \(z\) and such that for all \(y\in\partial A\) we have
\(\langle y-z,n\rangle\leq0\),
where \(n\) is the normal of \(H\).
Because \(B\) is bounded, there is an \(r\geq0\) such that \(z+rn\in\partial B\).
It is easy to see that \(p(z+rn)=z\).
That means \(p:\partial B\rightarrow\partial A\) is surjective.

Let \(x_1,x_2\in\partial B\).
For \(i=1,2\) denote \(z_i=p(x_i)\) and let \(H_i\) be the hyperplane with normal \(x_i-z_i\) which contains \(z_i\).
Then
\(
\langle z_2-z_1,x_1-z_1\rangle
\leq0
\)
because otherwise it is straightforward to find a \(t>0\) small enough with \((1-t)z_1+tz_2\in\cl A\) which is closer to \(x_1\) than \(z_1\), which leads to a contradiction to \(p(x_1)=z_1\).
Similarly we must have
\(
\langle z_1-z_2,x_2-z_2\rangle
\leq0
.
\)
We can conclude
\[
|z_1-z_2||x_1-x_2|
\geq
\langle z_1-z_2,x_1-x_2\rangle
=
\langle z_1-z_2,x_1-z_1\rangle
+
\langle z_2-z_1,x_2-z_2\rangle
+
\langle z_1-z_2,z_1-z_2\rangle
\geq
|z_1-z_2|^2
.
\]
This means that the map \(p:\partial B\rightarrow\partial A\) is \(1\)-Lipschitz,
and we obtain \cref{it_perimeterconvex} because the Hausdorff measure does not increase under \(1\)-Lipschitz maps by \cite[Theorem~2.8]{MR3409135}.

For every \(\lambda\geq0\) denote \(A_\lambda=\{x\in A:\dist(x,A^\comp)\geq\lambda\}\).
Then \(A_\lambda\) is convex and by \cite[Theorem~3.14]{MR3409135} and \cref{it_perimeterconvex} we have
\[
\lm{\{x\in A:0<\dist(x,A^\comp)\leq\varepsilon\}}
=
\int_0^\varepsilon
\sm{\partial A_\lambda}
\intd\lambda
\leq
\int_0^\varepsilon
\sm{\partial A}
\intd\lambda
=
\varepsilon
\sm{\partial A}
.
\]
\end{proof}

%% file: general.tex
We start with a couple of tools that are used for both maximal operators.

\begin{lem}\label{lem_isoperimetricconsequence}
Let \(X\subset\mathbb{R}^d\) be an open set with finite measure and finite perimeter which satisfies a relative isoperimetric inequality
and denote \(c=\sm{\partial X}^d/\lm X^{d-1}\).
Let \(0<\lambda\leq1-\varepsilon<1\) and let \(E\) be a measurable set such that \(\lambda\leq\lm{E\cap X}/\lm X\leq1-\varepsilon\).
Then
\[\sm{\mb E\cap X}\gtrsim c^{-1/d}\varepsilon^{d-1}\lambda^{\f{d-1}d}\sm{\partial X}.\]
\end{lem}

Note that \(c\) is invariant under scaling of \(X\).

\begin{proof}
We first prove
\begin{equation}
\label{eq_isopereps}
\lm{E\cap X}^{d-1}
\lesssim\varepsilon^{-(d-1)}
\sm{\mb E\cap X}^d.
\end{equation}
If \(\varepsilon\geq\f12\) then \cref{eq_isopereps} follows directly from the relative isoperimetric inequality for \(X\).
For \(\varepsilon<\f12\) we obtain \cref{eq_isopereps} it from the relative isoperimetric inequality as follows
\[
\sm{\mb E\cap X}^d\gtrsim\lm{X\setminus E}^{d-1}
\geq\varepsilon^{d-1}
\lm{X}^{d-1}
\geq\varepsilon^{d-1}
\lm{E\cap X}^{d-1}
.
\]
From \cref{eq_isopereps} we conclude
\[\varepsilon^{-(d-1)}\sm{\mb E\cap X}\gtrsim\lm{E\cap X}^{\f{d-1}d}\geq\lambda^{\f{d-1}d}\lm X^{\f{d-1}d}\geq c^{-1/d}\lambda^{\f{d-1}d}\sm{\partial X}.\]
\end{proof}

\begin{lem}[Boxing inequality, c.f.\ Theorem 3.1 in Kinnunen, Korte, Shanmugalingam, Tuominen \cite{MR2400262}]\label{lem_boxing}
Let \(E\subset\mathbb{R}^d\) be a set with finite measure that is contained in the union of a set \(\B\) of balls \(B\) with \(\lm{E\cap B}\leq\lm B/2\). 
Then there is a set \(\F\) of balls \(F\) with \(\lm{F\cap E}=\lm F/2\) which covers almost all of \(E\).
Furthermore, each \(F\in\F\) is contained in a ball \(B\in\B\).
\end{lem}

\begin{proof}
It suffices to show that for every ball \(B(x_1,r_1)\in\B\)
every Lebesgue point \(x\in\mi E\) with \(x\in B(x_1,r_1)\)
is contained in a ball \(F\subset B(x_1,r_1)\) with \(\lm{F\cap E}=\lm F/2\).
By assumption
\begin{align*}
\lm{E\cap B(x_1,r_1)}&\leq\f{\lm{B(x_1,r_1)}}2\\
\intertext{and since \(x\) is a Lebesgue point there is a ball \(B(x_0,r_0)\) with \(x\in B(x_0,r_0)\subset B(x_1,r_1)\) and}
\lm{E\cap B(x_0,r_0)}&\geq\f{\lm{B(x_0,r_0)}}2.
\end{align*}
Define \(x_t=(1-t)\cdot x_0+t\cdot x_1\) and \(r_t=(1-t)\cdot r_0+t\cdot r_1\)
so that \(t\mapsto B(x_t,r_t)\) is a continuous transformation of balls. 
That means there is a \(t\) with 
\[\lm{E\cap B(x_t,r_t)}=\f{\lm{B(x_t,r_t)}}2.\]
Since \(x\in B(x_0,r_0)\subset B(x_t,r_t)\subset B(x_1,r_1)\) that means we have found the right ball.
\end{proof}

We will prove a more specialized version of \Cref{lem_boxing}. 

\begin{lem}\label{cla_surfacedistanceboxingball}
Let \(X\) be an open cube or a ball in \(\mathbb{R}^d\) and \(E\) a set with \(\lm{E\cap X}\geq\lambda\lm X\). 
Then there is a cover \(\C\) of \(\partial X\setminus\mc E\)
consisting of balls \(C\) with \(\diam C\leq2\diam X\) and
\begin{equation}\label{eq_ballcontainsmuchboundary}
\smb{\mb E\cap\Bigl\{y\in C:\dist(y,X^\comp)> \f{\lambda\diam C}{4dc_d}\Bigr\}}\gtrsim\lambda^{\f{d-1}d}\sm{\partial C}
,
\end{equation}
where \(c_d=2^d\) if \(X\) is a ball and \(c_d=d^{\f d2}\sigma_d\) if \(X\) is a cube.
\end{lem}
\begin{figure}
\centering
\includegraphics{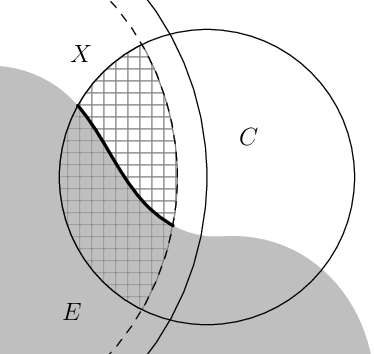}
\caption{The regions in \Cref{cla_surfacedistanceboxingball}.}
\end{figure}

The constants in \Cref{cla_surfacedistanceboxingball} are not important and
one could also impose a stronger bound on the diameter of the balls \(C\in\C\) for \(\lambda\) near 1.

\begin{proof}[Proof of \Cref{cla_surfacedistanceboxingball}]
It suffices to show that for each \(x\in\partial X\setminus\mc E\) there is a ball \(C\) centered in \(x\) that satisfies \cref{eq_ballcontainsmuchboundary}.
Let \(x\in\partial X\setminus\mc E\) and for \(0<r\leq\diam X\) define
\[
A(r)
=
\Bigl\{y\in B(x,r):\dist(y,X^\comp)> \f{\lambda r}{2dc_d}\Bigr\}
.
\]

We first show that \(A(r)\) is a John domain.
Consider the case that \(X\) is a ball.
Then there is a point \(z\in X\cap B(x,r)\) such that \(B(z,\f r2)\subset X\cap B(x,r)\).
That means
\[
B\Bigl(z,\f r4\Bigr)
\subset
B\Bigl(z,\f r2-\f{\lambda r}{2dc_d}\Bigr)
\subset
A(r)
.
\]
Now let \(X\) be a cube.
Then \(X\cap B(x,r)\) contains a cube with diameter at least \(r\),
i.e.\ sidelength at least \(\f r{\sqrt d}\).
Thus, \(A(r)\) contains a cube with sidelength at least
\[
\f r{\sqrt d}-2\f{\lambda r}{2dc_d}
\geq
\f r{\sqrt d}\Bigl(1-\f1{\sqrt dc_d}\Bigr)
\geq
\f r{2\sqrt d}
\]
which in turn contains a ball \(B\) with radius \(\f r{4\sqrt d}\).
The last inequality holds because \(1^{\f{1+1}2}\sigma_1=2\) and \(\sqrt dc_d=d^{\f{d+1}2}\sigma_d\) is increasing in \(d\) by \Cref{eq_unitvolumeratio}.
We have shown that there is a point \(z\in A(r)\) such that
\begin{equation}
\label{eq_Arlikeball}
B\Bigl(z,\f r{4\sqrt d}\Bigr)
\subset
A(r)
,
\end{equation}
both if \(X\) is a cube or a ball.
For any \(y\in A(r)\) we have \(\dist(y,z)\leq\diam(A(r))\leq 2r\).
Because \(A(r)\) is convex by \Cref{lem_surfaceconvex}\cref{it_intersectionconvex}, it contains the convex hull of \(B\bigl(z,\f r{4\sqrt d}\bigr)\cup\{y\}\).
We can conclude that \(A(r)\) is a John domain with \(K=\f{2r}{r/(4\sqrt d)}=8\sqrt d\).

We have
\begin{align}
\nonumber
\lm{B(x,r)\setminus A(r)}
&\leq
\lmb{\Bigl\{y:0<\dist(y,(B(x,r)\cap X)^\comp)\leq \f{\lambda r}{2dc_d}\Bigr\}}
\\
\nonumber
&\leq
\f{\lambda r}{2dc_d}\sm{\partial(B(x,r)\cap X)}
\\
\nonumber
&\leq
\f{\lambda r}{2dc_d}\sm{\partial B(x,r)}
\\
\nonumber
&=
\f\lambda{2c_d}\lm{B(x,r)}
\\
\label{eq_awayfromboundary}
&\leq
\f\lambda2\lm{B(x,r)\cap X}
,
\end{align}
where the last inequality holds because as observed above \(B(x,r)\cap X\) contains a ball with radius \(\f r2\) if \(X\) is a ball,
and a cube with sidelength \(\f r{\sqrt d}\) if \(X\) is a cube.
Then from
\[\f{\lm{X\cap E}}{\lm X}\geq\lambda\]
and \cref{eq_awayfromboundary} with \(r=\diam X\) we get
\[\f{\lm{A(\diam X)\cap E}}{\lm{A(\diam X)}}\geq\f{\lm{A(\diam X)\cap E}}{\lm X}\geq\lambda-\f\lambda2=\f\lambda2.\]
Since \(x\not\in\mc E\) we have \(\lm{E\cap B(x,r)}/r^d\rightarrow0\) for \(r\rightarrow0\).
By \cref{eq_Arlikeball} this implies that there is an \(r_0\) with
\[\f{\lm{A(r_0)\cap E}}{\lm{A(r_0)}}\leq\f\lambda2.\]
By continuity we conclude that there is an \(r_0\leq r\leq\diam X\) such that 
\[\f{\lm{A(r)\cap E}}{\lm{A(r)}} =\f\lambda2.\]
By \cref{eq_Arlikeball} and \Cref{lem_surfaceconvex}\cref{it_intersectionconvex} we have
\begin{equation}
\label{eq_dBdA}
\sm{\partial B(x,r)}
\lesssim
\smb{\partial B\Bigl(z,\f r{4\sqrt d}\Bigr)}
\leq
\sm{\partial A(r)}
.
\end{equation}
Because \(A(r)\) is a John domain it satisfies a relative isoperimetric inequality by \Cref{lem_isoperimetric},
so that we can apply \Cref{lem_isoperimetricconsequence} with \(X=A(r)\) and \(\varepsilon=\f12\)
and obtain
\begin{equation}
\label{eq_Aisoperimetric}
\sm{\partial A(r)}
\lesssim
\lambda^{-\f{d-1}d}\sm{\mb E\cap A(r)}
.
\end{equation}
Combining \cref{eq_dBdA} and \cref{eq_Aisoperimetric} we obtain \cref{eq_ballcontainsmuchboundary}, which finishes the proof. 
\end{proof}

Note that the following \Cref{lem_finMf} addresses the reduced boundary \(\rb E\) and not the measure theoretic boundary \(\mb E\).

\begin{lem}\label{lem_finMf}
Let \(\Omega\subset\mathbb{R}^d\) be an open set and let \(E\subset\mathbb{R}^d\) be measurable.
Then for every \(\lambda\in[0,1)\) and
for both the dyadic and the uncentered maximal operator with domain \(\Omega\)
we have
\(\mi E\cap\Omega\subset\{\M\ind E=1\}\).
For the uncentered maximal operator we furthermore have
\(\rb E\cap\Omega\subset\{\M\ind E=1\}\).
\end{lem}
This is a slightly more precise version of \(\M f\geq f\) almost everywhere for characteristic functions.

\begin{proof}
Let \(x\in\mi E\cap\Omega\).
Then for every \(\varepsilon>0\) there is a ball \(B\subset\Omega\) with center \(x\) and with \(\lm{B\setminus E}\leq\varepsilon\lm{B}\) and a dyadic cube \(Q\) with \(x\in Q\subset B\) and \(\lm{Q}\gtrsim\lm{B}\). 
That means \(\lm{Q\setminus E}\leq\varepsilon\lm{B}\lesssim\varepsilon\lm{Q}\).
We can conclude \(\M\ind E(x)=1\).

Let \(x\in\rb E\cap\Omega\).
It suffices to consider \(x=0\) and
\[\lim_{r\rightarrow0}\avint_{B(0,r)}\nu_E=(1,0,\ldots,0).\]
Then for \(r\) small enough we have \(0\in B_r=B((-r,0,\ldots,0),r+r^2)\subset\Omega\), and so by \Cref{lem_blowuprb} we obtain
\[
\lim_{r\rightarrow0}
\avint_{B_r}\ind E
=
\lim_{r\rightarrow0}
\f{\lm{\{y\in B_r:y_1<0\}}}{\lm{B_r}}
=
\lim_{r\rightarrow0}
\f{\lm{\{y\in B(0,r+r^2):y_1<r\}}}{\lm{B(0,r+r^2)}}
=1
.
\]
\end{proof}

%% file: characteristicf_levelsets.tex
In this section we discuss the argument for the dyadic maximal operator.
It already showcases the main idea of the proof for the uncentered maximal operator.
For the superlevelset of the dyadic maximal operator we have
\[\{\M\ind E>\lambda\}=\bigcup\{\tx{dyadic cube }Q:\lm{E\cap Q}>\lambda\lm Q\}.\]
The first step in the proof of \Cref{eq_levelsetsdyadic} is
to consider only a finite set \(\Q\) of cubes \(Q\) with \(\lm{E\cap Q}>\lambda\lm Q\) instead of the whole set,
because this allows to write
\[\sm{\mb\bigcup\Q}\leq\sum_{Q\in\Q}\sm{\partial Q\cap\mb\bigcup\Q}.\]
From there we use approximation results to extend to the union of all cubes \(Q\) with \(\lm{E\cap Q}>\lambda\lm Q\).
The strategy for the uncentered maximal operator is similar,
but with cubes replaced by balls.
The main argument is \Cref{lem_varav},
which is more or less \Cref{eq_levelsetsdyadic} for the case that \(\{\M\ind E>\lambda\}\) consists of only one cube. 

\begin{pro}\label{lem_varav} Let \(0<\lambda\leq1\), let \(Q\) be a cube and let \(E\subset\mathbb{R}^d\) be a measurable set with \(\lm{E\cap Q}\geq\lambda\lm Q\). 
Then
\[\sm{\partial Q\setminus\mc E}\lesssim\lambda^{-\f{d-1}d}\sm{\mb E\cap\ti Q}.\]
\end{pro}

\begin{proof}[Proof of \Cref{lem_varav}]
We apply \Cref{cla_surfacedistanceboxingball} to \(X=\ti Q\)
and for the resulting cover use \Cref{lem_disjointcover} to extract a disjoint subcollection \(\C\) 
such that \(5\C\) still covers \(\partial Q\setminus\mc E\). 
Then
by \Cref{lem_surfaceconvex}\cref{it_intersectionconvex,it_perimeterconvex}
and \Cref{cla_surfacedistanceboxingball}
we have
\begin{align*}
\sm{\partial Q\setminus\mc E}
&\leq
\sum_{C\in\C}\sm{\partial Q\cap5C}
\\
&\leq
\sum_{C\in\C}\sm{\partial 5C}
\\
&\lesssim\lambda^{-\f{d-1}d}\sum_{C\in\C}\sm{\mb E\cap C\cap\ti Q}\\
&\leq\lambda^{-\f{d-1}d}\sm{\mb E\cap\ti Q}.
\end{align*}
\end{proof}

\begin{rem}\label{rem_varavball}
For \(\lambda\leq\f12\) \Cref{lem_varav} also follows directly from the relative isoperimetric inequality \cref{eq_isoperimetric} for \(Q\).
\Cref{lem_varav} also holds for \(Q\) being a ball.
\end{rem}

\begin{proof}[Proof of \Cref{eq_levelsetsdyadic}]
For each \(x\in\{\M\ind E>\lambda\}\cap \Omega\) there is a dyadic cube \(Q\subset\Omega\) with \(x\in Q\) and \(\lm{E\cap Q}>\lambda\lm Q\). 
Since there are only countably many dyadic cubes we can enumerate them as \(Q_1,Q_2,\ldots\). 
For each \(n\) let
\[
\Q_n
=
\{Q_i:\forall j=1,\ldots,n\tx{ with }j\neq i\tx{ we have }Q_i\not\subset Q_j\}
.
\]
Then \(\bigcup\Q_n=Q_1\cup\ldots\cup Q_n\) and thus
\[
\bigcup_nQ_n=\{\M\ind E>\lambda\}
.
\]
Because \(E\) and \(\mi E\) agree up to measure zero
and
\(\mi E\subset\{\M\ind E>\lambda\}\)
by \Cref{lem_finMf},
we have that \(\bigcup\Q_n\cup E\) converges to \(\{\M\ind E>\lambda\}\) in \(L^1_\loc(\Omega)\).
Therefore, by \Cref{lem_l1approx,lem_boundaryofunion} we obtain
\begin{align}
\nonumber\sm{\mb{\{\M\ind E>\lambda\}}\cap \Omega}&\leq\limsup_{n\rightarrow\infty}\sm{\mb{(\bigcup\Q_n\cup E)}\cap \Omega}\\
\label{eq_summandtoomany}&\leq\limsup_{n\rightarrow\infty}\sm{(\mb{\bigcup\Q_n}\setminus\mc E)\cap \Omega}+\sm{\mb E\cap \Omega}.
\end{align}
It is not necessary,
but in the line corresponding to \cref{eq_summandtoomany} in the proof for the uncentered Hardy-Littlewood maximal function
we can actually eliminate the term \(\sm{\mb E\cap \Omega}\)
thanks to \Cref{lem_finMf};
see \cref{eq_summandtoomanyagain} in \Cref{sec_uncentered} and the subsequent comment.
Here this is not so clear because for the dyadic maximal function \Cref{lem_finMf} is weaker.
But in any case, it suffices to estimate the first term on the right hand side of \cref{eq_summandtoomany}.
We invoke \Cref{lem_varav} and use that the cubes in \(\Q_n\) are disjoint and obtain
\begin{align*}
\sm{(\mb{\bigcup\Q_n}\setminus\mc E)\cap \Omega}&\leq\sum_{Q\in\Q_n}\sm{(\mb Q\setminus\mc E)\cap \Omega}\\
&\lesssim\sum_{Q\in\Q_n}\lambda^{-\f{d-1}d}\sm{\mb E\cap Q}\\
&\leq\lambda^{-\f{d-1}d}\sm{\mb E\cap \Omega\cap\{\M\ind E>\lambda\}}.
\end{align*}
\end{proof}

\Cref{lem_varav} readily implies \Cref{eq_levelsetsdyadic} because \(\{\M\ind E>\lambda\}\) is a disjoint union of such cubes.
Two balls however can have nontrivial intersections,
which is why the proof for the uncentered Hardy-Littlewood maximal operator is much more complicated than the proof for the dyadic maximal operator.

%% file: uncentered.tex
In this section we prove \Cref{eq_levelsets}.
The main step is \Cref{pro_levelsets_finite_g}.
It is \Cref{lem_varav} for a set \(\B\) of finitely many balls \(B\) with \(\lm{B\cap E}>\lambda\lm B\) instead of one cube.
\Cref{pro_levelsets_finite_g} comes with an additional but harmless factor \((1-\log\lambda)\).
We will show in \Cref{sec_optimal} that this factor can be removed.

\begin{lem}\label{cla_largeboundaryinball}
Let \(K>0\), let \(C\) be a ball and let be \(\B\) a finite set of balls \(B\) with \(\diam(B)\geq K\diam(C)\). 
Then
\[\sm{\mb{\bigcup\B}\cap C}\lesssim (K^{-d}+1)\sm{\partial C}.\]
\end{lem}

\begin{figure}
\centering
\includegraphics{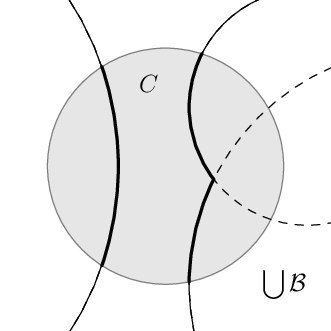}
\caption{The objects in \Cref{cla_largeboundaryinball}.}
\end{figure}
The rate \(K^{-d}\) does not play a role in the application.

\begin{proof}
By translation and scaling it suffices to consider the case \(C=B(0,1)\).
Let \(B(x,r)\) be a ball with \(|x|\geq4d+1\) whose boundary intersects \(B(0,1)\),
which means \(4d<r<4d+2\).
For any point \(y=(y_1,\ldots,y_d)\in\mathbb{R}^d\) denote \(\proj y=(y^1,\ldots,y^{d-1})\).
Assume that \(|x^d|=\max\{|x_1|,\ldots,|x_d|\}\), so that
\[
|\proj x|^2=|x|^2-x_d^2\leq\Bigl(1-\f1d\Bigr)|x|^2
.
\]
Then for every \(y\in B(0,1)\) we have
\begin{align*}
|\proj y-\proj x|
&\leq
|\proj x|+1
\leq
\sqrt{1-\f1d}|x|+1
\leq
\Bigl(1-\f1{2d}\Bigr)(r+1)+1
=r
\Bigl(1+\f2r-\f1{2d}-\f1{2dr}\Bigr)
\leq
r
,
\end{align*}
and
\[
x_d-y_d
\geq
\sqrt{\f1d}|x|
-1
\geq
\f{r-(\sqrt d+1)}{\sqrt d}
>0
.
\]
Therefore the function
\[
y\mapsto\varphi(\hat y)
=
x_d-\sqrt{r^2-|\proj{y-x}|^2}
\]
is well defined for \(y\in B(0,1)\),
we have
\[
B(x,r)\cap B(0,1)
=
\{y\in B(0,1):y_d>\varphi(\hat y)\}
,
\]
and for \(y\in\partial B(x,r)\cap B(0,1)\) the gradient of \(\varphi\) in \(y_1,\ldots,y_{d-1}\) is bounded by
\[
|\nabla\varphi(\hat y)|
=
\f{|\proj{x-y}|}{\sqrt{r^2-|\proj{x-y}|^2}}
=
\f{|\proj{x-y}|}{|x_d-y_d|}
\leq
\f{\sqrt dr}{r-(\sqrt d+1)}
\leq
\f{4d^{\f32}}{4d-(\sqrt d+1)}
\leq
2\sqrt d
.
\]
For the case that all balls \(B\in\B\) have radius at least \(4d\)
we can conclude that the boundary of the union of all balls of the above form is a piece of the infimum of \(2\sqrt d\)-Lipschitz graphs, and thus itself a piece of a \(2\sqrt d\)-Lipschitz graph.
We can conclude that
\begin{align*}
\smb{\partial
\bigcup\bigl\{B(x,r)\in\B:x_d=\max\{|x_1|,\ldots,|x_d|\}\bigr\}\cap B(0,1)
}
&\leq
\sqrt{4d+1}\sigma_{d-1}
\\
&=
\f{\sqrt{4d+1}\sigma_{d-1}}{d\sigma_d}
\sm{\partial B(0,1)}
.
\end{align*}
By rotation we obtain the same bound for the union of those balls \(B(x,r)\in\B\) with \(\pm x_i=\max\{|x_1|,\ldots,|x_d|\}\) for any \(i=1,\ldots,d\) and any sign.
This finishes the proof for \(K\geq4d\).

If \(K<4d\) then we cover \(B(0,1)\) by \(\lesssim\bigl(\f{4d}K\bigr)^d\) many balls \(C_1,C_2,\ldots\) so that for each \(i\) we have \(\diam(B)\geq 4d\diam(C_i)\). 
Then
\[
\sm{\mb{\bigcup\B}\cap B(0,1)}
\leq
\sum_i \sm{\mb{\bigcup\B}\cap C_i}
\lesssim
\sum_i \sm{\partial C_i}
\lesssim
\Bigl(\f{4d}K\Bigr)^d\sm{\partial B(0,1)}.
\]
\end{proof}

\input{surface_boxing.tex}

Now we extend \Cref{pro_levelsets_finite_g} to the whole set \(\{\M\ind E>\lambda\}\). 

\begin{proof}[Proof of \Cref{eq_levelsets}]
Note that
\[\{\M\ind E>\lambda\}=\bigcup\{B\subset \Omega:\lm{B\cap E}>\lambda\lm B\}.\]
First we pass to a countable set of balls.
By the Lindel\"of property, for example Proposition~1.5 in \cite{MR2867756},
there is a sequence of balls with
\[\{\M\ind E>\lambda\}= B_1\cup B_2\cup\ldots\]
such that for each \(i\) we have \(\lm{E\cap B_i}>\lambda\lm{B_i}\). 
Denote \(\B_n=\{B_1,\ldots,B_n\}\). 
Then \(\bigcup\B_n\) converges to \(\{\M\ind E>\lambda\}\) in \(L^1_\loc(\Omega)\).
Furthermore, by \Cref{lem_finMf} we have
\[
\bigcup\B_n
\subset
\bigcup\B_n\cup\mi E
\subset
\{\M\ind E>\lambda\}
,
\]
which means that also \(\bigcup\B_n\cup E\) converges to \(\{\M\ind E>\lambda\}\) in \(L^1_\loc(\Omega)\).
Since \(E\) and \(\mi E\) agree up to a set of measure zero we have \(\mc{(\mi E)}=\mc E\) and \(\mb{(\mi E)}=\mb E\).
We apply the approximation using \Cref{lem_l1approx} and then divide the boundary using \Cref{lem_boundaryofunion} and obtain
\begin{align}
\nonumber\sm{\mb{\{\M\ind E>\lambda\}}\cap \Omega}&\leq\limsup_{n\rightarrow\infty}\sm{\mb{(\bigcup\B_n\cup\mi E)}\cap \Omega}\\
\label{eq_summandtoomanyagain}&\leq\limsup_{n\rightarrow\infty}\sm{\mb{\bigcup\B_n}\setminus\mc E\cap \Omega}+\smb{\mb E\setminus\mib{\bigcup\B_n}\cap \Omega}.
\end{align}
By \Cref{lem_finMf} the second summand is bounded by \(\sm{\mb E\cap \Omega\cap \{\M\ind E>\lambda\}}\).
In fact, if \(\sm{\mb E\cap \Omega\cap \{\M\ind E>\lambda\}}\) is finite
then the second summand in \cref{eq_summandtoomanyagain} even goes to \(0\) for \(n\rightarrow\infty\).
This is due to \Cref{lem_finMf} for the uncentered maximal function,
because
\[\mib{\bigcup\B_n}\supset\bigcup\B_n\]
which is an increasing sequence in \(n\) which exhausts \(\{\M\ind E>\lambda\}\).
In any case, it remains to estimate the first summand in \cref{eq_summandtoomanyagain}
which we do using \Cref{pro_levelsets_finite_g}
\begin{align*}
\sm{\mb{\bigcup\B_n}\setminus\mc E}&\lesssim\lambda^{-\f{d-1}d}(1-\log\lambda)\sm{\mb E\cap\bigcup\B_n}\\
&\leq\lambda^{-\f{d-1}d}(1-\log\lambda)\sm{\mb E\cap\{\M\ind E>\lambda\}}.
\end{align*}
\end{proof}

%% file: surface_boxing.tex
In this section for a set of balls \(\B\) we denote by \(\B_n\) the set of those \(B\in\B\) with \(\diam(B)\in[\f12,1)2^n\).
Further define \(\B_{>n}=\bigcup_{k>n}\B_k\) and \(\B_{\geq n},\B_{<n},\ldots\) accordingly.

\begin{lem}\label{cla_surfacedistanceboxing}
Let \(\lambda\in(0,1)\), let \(E\subset\mathbb{R}^d\) be measurable and let \(\B\) be a finite set
of balls \(B\) with \(\lm{E\cap B}>\lambda\lm B\). 
Then there is a set of balls \(\C\) such that for each \(n\in\mathbb{Z}\) the following holds.
\begin{enumerate}
\item\label{it_disjointscale} The balls in \(\C_n\) are disjoint.
\item\label{it_coverscale} The boundary piece \(\mb{\bigcup\B}\cap\mb{\bigcup\B_{n-1}}\setminus\mc E\) is covered by \(5\C_{\leq n}\).
\item\label{it_closetoboundary} Each \(C\in\C_n\) has distance at most \(2\diam(C)\) to \(\mb{\bigcup\B}\setminus\mc E\).
\item
We have
\(
\smb{\mb E\cap\bigl\{x\in C:\dist(x,(\bigcup\B)^\comp)\geq\lambda d^{-1}2^{n-d-2}\bigr\}}
\gtrsim
\lambda^{\f{d-1}d}
\sm{\partial C}
.
\)
\label{it_enoughvardist}
\end{enumerate}
\end{lem}

\begin{proof}
Apply \Cref{cla_surfacedistanceboxingball} to each ball in \(\B\) and denote by \(\tilde\C\) the union of all of these balls. 
They cover \(\mb{\bigcup\B}\setminus\mc E\). 
In particular \(\mb{\bigcup\B}\cap\mb{\bigcup\B_{n-1}}\setminus\mc E\) is covered by \(\tilde\C_{\leq n}\). 
Let \(n\in\mathbb{Z}\). 
By \Cref{lem_disjointcover} there is a subcollection \(\C_n\) of \(\tilde\C_n\) of disjoint balls with \(\bigcup\tilde\C_n\subset\bigcup5\C_n\). 
That means \Cref{it_disjointscale} and \Cref{it_coverscale} are satisfied. 
Now remove those balls \(C\) from \(\C_n\) such that \(5C\) does not touch \(\mb{\bigcup\B}\setminus\mc E\). 
Then \Cref{it_coverscale} still holds and we also get \Cref{it_closetoboundary}.

Let \(C\in\C_n\) and let \(B\in\B\) be the ball which gave rise to \(C\). 
We use \(B\subset\bigcup\B\)
and \Cref{cla_surfacedistanceboxingball}
to obtain
\begin{align*}
&\quad\sm{\mb E\cap\{x\in C:\dist(x,\bigcup\B^\comp)>\lambda d^{-1}2^{n-d-2}\}}\\
&\geq\smb{\mb E\cap\Bigl\{x\in C:\dist(x,\bigcup\B^\comp)>\f{\lambda\diam C}{2^{d+2}d}\Bigr\}}\\
&\geq\smb{\mb E\cap\Bigl\{x\in C:\dist(x,B^\comp)>\f{\lambda\diam C}{2^{d+2}d}\Bigr\}}\\
&\gtrsim\lambda^{\f{d-1}d}\sm{\partial C},
\end{align*}
proving \Cref{it_enoughvardist}.
\end{proof}

\begin{pro}\label{pro_levelsets_finite_g} Let \(\lambda\in(0,1)\). 
Let \(E\subset\mathbb{R}^d\) be a set of locally finite perimeter and let \(\B\) be a finite set of balls such that for each \(B\in\B\) we have \(\lm{E\cap B}>\lambda\lm B\).
Then
\[\sm{\mb{\bigcup\B}\setminus\mc E}\lesssim\lambda^{-\f{d-1}d}(1-\log\lambda)\smb{\mb E \cap\mib{\bigcup\B}}.\]
\end{pro}

\Cref{pro_levelsets_finite_g} is the key ingredient in the proof of \Cref{eq_levelsets}.
The idea of the proof of \Cref{pro_levelsets_finite_g} is that we want to split \(\mb{\bigcup\B}\) into pieces
according to how far away a piece of \(\mb{\bigcup\B}\) is from a significant portion of \(E\),
and then identify for each such piece of \(\mb{\bigcup\B}\) a corresponding piece of \(\mb E\) with comparable size. 

\begin{proof}[Proof of \Cref{pro_levelsets_finite_g}]
We use \Cref{cla_surfacedistanceboxing}. 
We first rearrange \(\mb\bigcup\B\setminus\mc E\) and divide it according to the \((\C_n)_n\) in \Cref{cla_surfacedistanceboxing} and apply \Cref{cla_largeboundaryinball}.
We obtain
\begin{align*}
\sm{\mb{\bigcup\B}\setminus\mc E}&=\smb{\bigcup_k\mb{\bigcup\B}\cap\mb{\bigcup\B_k}\setminus\mc E}\\
&=\smb{\bigcup_k\mb{\bigcup\B}\cap\mb{\bigcup\B_k}\cap\bigcup_{n\leq k+1}\bigcup5\C_n}\\
&=\smb{\bigcup_n\bigcup_{k\geq n-1}\mb{\bigcup\B}\cap\mb{\bigcup\B_k}\cap\bigcup5\C_n}\\
&=\smb{\bigcup_n\mb{\bigcup\B}\cap\mb{\bigcup\B_{\geq n-1}}\cap\bigcup5\C_n}\\
&\leq\sum_n\sm{\mb{\bigcup\B}\cap\mb{\bigcup\B_{\geq n-1}}\cap\bigcup5\C_n}\\
&\leq\sum_n\sum_{C\in\C_n}\sm{\mb{\bigcup\B}\cap\mb{\bigcup\B_{\geq n-1}}\cap5C}\\
&\lesssim\sum_n\sum_{C\in\C_n}\sm{\partial C}.
\end{align*}
In what follows we apply first \Cref{it_enoughvardist}, then \Cref{it_disjointscale} and \Cref{it_closetoboundary}.
We obtain
\begin{align*}
\sum_{C\in\C_n}\sm{\partial C}
&\lesssim\lambda^{-\f{d-1}d}\sum_{C\in\C_n}\sm{\mb E\cap\{x\in C:\dist(x,\bigcup\B^\comp)\geq\lambda d^{-1}2^{n-d-2}\}}\\
&=\lambda^{-\f{d-1}d}\sm{\mb E\cap\{x\in\bigcup\C_n:\dist(x,\bigcup\B^\comp)\geq\lambda d^{-}2^{n-d-2}\}}\\
&\leq\lambda^{-\f{d-1}d}\sm{\mb E\cap\{x:\lambda d^{-1}2^{n-d-2}\leq\dist(x,\bigcup\B^\comp)\leq2^{n+1}\}}.
\end{align*}
Now we sum over \(n\).
Since for a fixed number \(r\in\mathbb{R}\) the condition
\(\lambda d^{-1}2^{n-d-2}\leq r\leq2^{n+1}\)
can only occur for \(d+3+\log_2d-\log_2\lambda\) many \(n\in\mathbb{Z}\),
we can bound
\begin{align*}
\sm{\mb{\bigcup\B}\setminus\mc E}
&\lesssim
\lambda^{-\f{d-1}d}\sum_n\sm{\mb E\cap\{x:\lambda d^{-1}2^{n-d-2}\leq\dist(x,\bigcup\B^\comp)\leq2^{n+1}\}}\\
&\lesssim\lambda^{-\f{d-1}d}(1-\log\lambda)\sm{\mb E\cap\bigcup\B}.
\end{align*}

\end{proof}

\begin{rem}
If the balls in \(\bigcup_n\C_n\) were disjoint then we could get rid of the factor \(1-\log\lambda\) by using \Cref{rem_varavball} instead of \Cref{it_enoughvardist}.
\end{rem}

%% file: optimal.tex
In this section we prove the following improvement of \Cref{eq_levelsets}.

\begin{pro}\label{eq_levelsets_optimal}
Let \(\M\) be the local uncentered maximal operator. 
Let \(E\subset\mathbb{R}^d\) be a set with locally finite perimeter and let \(\lambda\in(0,1)\).
Then 
\[\sm{\mb{\{\M\ind E> \lambda\}}\cap\Omega}\lesssim\lambda^{-\f{d-1}d}\sm{\mb E\cap\{\M\ind E>\lambda\}}.\]
\end{pro}

More important than the statement of \Cref{eq_levelsets_optimal} is maybe the proof strategy.
It may be helpful when attempting to generalize \Cref{theo_goal} to
\(\var\M f\lesssim\var f\) for general functions \(f\) with bounded variation.

\begin{rem}
From taking \(\Omega=\mathbb{R}^d\) and \(E=B(0,1)\) it follows that the rate \(\lambda^{-\f{d-1}d}\) in \Cref{eq_levelsets_optimal} is optimal.
\end{rem}

In order to prove \Cref{eq_levelsets_optimal} it suffices to prove the following improvement of \Cref{pro_levelsets_finite_g}.
\begin{pro}\label{pro_levelsets_finite_l_local}
Let \(\lambda\in[0,1/2)\), let \(E\subset\mathbb{R}^d\) be a set of locally finite perimeter and let \(\B\) be a finite set of balls such that for each \(B\in\B\) we have \(\lambda\lm B<\lm{E\cap B}\leq\f12\lm B\).
Then
\[\sm{\mb{\bigcup\B} }\lesssim\lambda^{-\f{d-1}d}\sm{\mb E \cap\bigcup\B }.\]
\end{pro}

\begin{proof}[Proof of \Cref{eq_levelsets_optimal}]
Let \(\B\) be a finite set of balls \(B\) with \(\lm{B\cap E}\geq\lambda\lm B\).
Then
\begin{align*}
\sm{\partial\bigcup\B\setminus\mc E}
&\leq\sm{\partial\{B\in\B:\lm{B\cap E}>\lm B/2\}\setminus\mc E}\\
&+\sm{\partial\{B\in\B:\lambda\lm B<\lm{B\cap E}\leq\lm B/2\}\setminus\mc E}
\end{align*}
By \Cref{pro_levelsets_finite_g} the first summand in the previous display bounded by a dimensional costant times \(\sm{\mb E\cap\bigcup\B}\)
and by \Cref{pro_levelsets_finite_l_local} the second summand is bounded by a dimensional constant times \(\lambda^{-\f{d-1}d}\sm{\mb E\cap\bigcup\B}\).
We conclude
\[\sm{\partial\bigcup\B\setminus\mc E}\lesssim\lambda^{-\f{d-1}d}\sm{\mb E\cap\bigcup\B},\]
which is \Cref{pro_levelsets_finite_g} without the factor \(1-\log\lambda\).
Now we can repeat the proof of \Cref{eq_levelsets} verbatim without the factor \(1-\log\lambda\).
\end{proof}

There is a weaker version of \Cref{pro_levelsets_finite_l_local} which has a simpler proof,
but already suffices to prove \Cref{eq_levelsets_optimal} for \(\Omega=\mathbb{R}^d\).

\begin{pro}\label{pro_levelsets_finite_l}
There is an \(\varepsilon>0\) depending only on the dimension such that for all \(\lambda\in[0,\varepsilon)\) the following holds. 
Let \(E\subset\mathbb{R}^d\) be a set of locally finite perimeter and let \(\B\) be a finite set of balls such that for each \(B\in\B\) we have \(\lambda\lm B<\lm{E\cap B}\leq\varepsilon\lm B\).
Then there is a finite superset \(\tilde\B\) of \(\B\) consisting of balls \(B\) with \(\lm{E\cap B}>\lambda\lm B\) that satisfies
\[\sm{\mb{\bigcup\tilde\B} }\lesssim\lambda^{-\f{d-1}d}\sm{\mb E \cap\bigcup\B }.\]
\end{pro}

\begin{proof}[Proof of \Cref{eq_levelsets_optimal} for \(\Omega=\mathbb{R}^d\)]
Take \(\varepsilon>0\) from \Cref{pro_levelsets_finite_l}.
For \(\lambda\geq\varepsilon\) \Cref{eq_levelsets_optimal} already follows from \Cref{eq_levelsets}.
It suffices to consider the case that there is an \(x_0\in\mathbb{R}^d\) with \(\lambda<\M\ind E(x_0)\leq\varepsilon\).
Let \(x\in\mathbb{R}^d\) with \(\M\ind E(x)>\lambda\).
Then there is a ball \(C\ni x\) with \(\lm{E\cap C}>\lambda\lm C\),
while \(\lm{E\cap B(x_0,|x-x_0|+1)}\leq\varepsilon\lm{B(x_0,|x-x_0|)+1}\).
By continuously transforming \(C\) into \(B(x_0,|x-x_2|+1)\) we can conclude that
\(\{\M\ind E>\lambda\}\) is a union of balls \(B\) with \(\lambda\lm B<\lm{E\cap B}\leq\varepsilon\lm B\).
Thus by the Lindel\"of property there is a sequence of balls \((B_n)_n\)
with \(\lambda\lm{B_n}<\lm{E\cap B_n}\leq\varepsilon\lm{B_n}\)
such that \(\{\M\ind E>\lambda\}=B_1\cup B_2\cup\ldots\).
Let \(\tilde\B_n\) be the finite superset of \(\B_n=\{B_1,\ldots,B_n\}\) from \Cref{pro_levelsets_finite_l}.
Then 
\[\bigcup\B_n\subset\bigcup\tilde\B_n\subset\{\M\ind E>\lambda\}\]
which means that \(\bigcup\tilde\B_n\) converges to \(\{\M\ind E>\lambda\}\) in \(L^1_\loc(\Omega)\).
Thus we get as in the proof of \Cref{eq_levelsets} that
\[
\sm{\mb{\{\M\ind E>\lambda\}}}\leq\limsup_{n\rightarrow\infty}\sm{\mb{\bigcup\tilde\B_n}}.
\]
By \Cref{pro_levelsets_finite_l} we have
\begin{align*}
\sm{\mb{\bigcup\tilde\B_n}}
&\lesssim\lambda^{-\f{d-1}d}\sm{\mb E\cap\bigcup\B_n}\\
&\leq\lambda^{-\f{d-1}d}\sm{\mb E\cap\{\M\ind E>\lambda\}}.
\end{align*}
\end{proof}

\subsection{The global case \texorpdfstring{\(\Omega=\mathbb{R}^d\)}{Omega=Rd}}

In this subsection we present a proof of \Cref{pro_levelsets_finite_l}.
It already contains some of the ideas for the general local case \Cref{pro_levelsets_finite_l_local}.

\input{small_average.tex}

\subsection{The general local case \texorpdfstring{\(\Omega\subset\mathbb{R}^d\)}{Omega is a subset of Rd}}

In this subsection we present a proof of \Cref{pro_levelsets_finite_l_local}.
It requires a few more steps than the proof of \Cref{pro_levelsets_finite_l}.

\input{small_average_local.tex}

%% file: small_average.tex
\begin{proof}[Proof of \Cref{pro_levelsets_finite_l}]
Restrict \(\varepsilon\leq\f12\). 
Let \(\F'\) be the collection of balls from \Cref{lem_boxing} applied to \(E\cap\bigcup\B\) and \(\B\). 
Let \(\tilde\F\) be the countable disjoint subcollection from \Cref{lem_disjointcover}. 
Extract from that a finite subcollection \(\F\) so that for every \(B\in\B\) we have
\begin{equation}\label{eq_enoughfinite}
\lm{E\cap\bigcup5\F\cap B}\geq\f\lambda2\lm{B}.
\end{equation}
This is possible since \(\B\) is finite. 
Here \(\F\) serves as a decomposition of \(E\) into pieces \(F\cap E\) where each piece has a substantial amount of boundary. 
The overall goal is to collect for each \(F\) its contribution to \(\sm{\mb{\bigcup\B}}\) and show that it is bounded by \(\partial F\).
First we enlarge \(\B\).
For every \(F\in\F\) the ball \(B=(2\lambda)^{-\f1d} F\) satisfies 
\[\lm{E\cap B}\geq\lm{E\cap F}=\f{\lm F}2=\lambda\lm B.\]
Add all those balls \(B\) to \(\B\). 
Then \(\B\) is still finite. 

Restrict \(\varepsilon\leq\f1210^{-d}\) and let \(r>0\) and \(F\in\F\) with \(\diam F\geq8r(2\lambda)^{\f1d}\).
Since we assume \(\lambda\leq\varepsilon\), we obtain
\[
\diam((2\lambda)^{-\f1d}F)
-
\diam(5F)
=
((2\lambda)^{-\f1d}-5)\diam F
\geq
(1-5(2\lambda)^{\f1d})\cdot8r
\geq
4r
,
\]
which means that any ball \(B\in\B\) with diameter at most \(r\) that intersects \(5F\) is entirely contained in \((2\lambda)^{-\f1d} F\in\B\).
Hence we may remove \(B\) from \(\B\) without changing \(\mb{\bigcup\B}\setminus\mc E\).
Conversely, we may assume that if \(B\in\B\) has diameter \(r\) and \(F\in\F\) is a ball for which \(5F\) intersects \(B\), then \(\diam F<8r(2\lambda)^{\f1d}\).
We further restrict \(\varepsilon\leq\f1420^{-d}\) and obtain
\begin{equation}\label{eq_bfintersect}
\f{\lm{5F}}{\lm B}<\f{5^d8^dr^d2\varepsilon}{2^dr^d}\leq\f12.
\end{equation}

For each \(n\in\mathbb{Z}\) denote by \(\B_n\) the set of balls in \(\B\) with \(\diam B\in[\f12,1)2^n\).
Denote by \(\F_n\) the set of those balls with \(\diam F\in2^n(2\lambda)^{\f1d}[4,8)\). 
Let \(B\in\B_n\) and let \(F\in\F\) such that \(5F\) intersects \(B\).
Then
\begin{equation}\label{eq_FinFntoo}
F\in\F_k\qquad\tx{for some }k\leq n.
\end{equation}
By \cref{eq_bfintersect},
any \(F\in\F\) such that \(5F\) intersects \(B\) is contained in \(3B\).
Thus we get from \cref{eq_enoughfinite,eq_FinFntoo} that
\[\f\lambda2\lm{B}\leq\sum_{F\in\F_{\leq n},\ F\subset3B}\lm{5F\cap B}.\]
We rewrite the previous display as
\begin{align}
\nonumber\sm{\partial B}
&\leq2
\sum_{k\leq n}\sum_{F\in\F_k,F\subset3B}
\f{\lm{5F\cap B}}{\lambda\lm{B}}\sm{\partial B}
\\
\nonumber
&\lesssim
\sum_{k\leq n}\sum_{F\in\F_k,F\subset3B}
\Bigl(\f{\lm F}{\lambda\lm{B}}\Bigr)^{\f1d}
\lambda^{-\f{d-1}d}
\Bigl(\f{\lm F}{\lm{B}}\Bigr)^{\f{d-1}d}
\sm{\partial B}
\\
\label{eq_dbtodf}
&\sim
\sum_{k\leq n}\sum_{F\in\F_k,F\subset3B}2^{k-n}\lambda^{-\f{d-1}d}\sm{\partial F}.
\end{align}

This estimate can be seen as a way to distribute \(\sm{\partial B}\) over the balls \(F\) that it contains. 
The next step will be to turn this dependence around, and see for a fixed \(F\) for how much variation of \(\sm{\mb{\bigcup\B}}\) it is responsible.

Since \(\B_n\) is finite we have
\[\sm{\mb{\bigcup\B_n}}=\sum_{B\in\B_n}\sm{\partial B\cap\mb{\bigcup\B_n}}.\]
We again multiply each summand by a number bounded from below according to \cref{eq_dbtodf},
\begin{align*}
\sum_{B\in\B_n}\sm{\partial B\cap\mb{\bigcup\B_n}}
&\lesssim\sum_{B\in\B_n}\f{\sm{\partial B\cap\mb{\bigcup\B_n}}}{\sm{\partial B}}\sum_{k\leq n}\sum_{F\in\F_k,F\subset3B}2^{k-n}\lambda^{-\f{d-1}d}\sm{\partial F}\\
&=\lambda^{-\f{d-1}d}\sum_{k\leq n}2^{k-n}\sum_{F\in\F_k}\sm{\partial F}\sum_{B\in\B_n,3B\supset F}\f{\sm{\partial B\cap\mb{\bigcup\B_n}}}{\sm{\partial B}}.
\end{align*}
Now we have reorganized \(\mb{\bigcup\B_n}\) according to the balls in \(\F\). 
We want to bound the contribution of each ball \(F\in\F\) uniformly.
For each \(F\in\F_k\) for which there is a ball \(B\in\B_n\) with \(F\subset3B\), denote by \(B_F\) a largest such ball \(B\).
Then for each \(B\in\B_n\) with \(F\subset3B\) have \(B\subset9B_F\). 
Thus
\[\sum_{B\in\B_n,3B\supset F}\f{\sm{\partial B\cap\mb{\bigcup\B_n}}}{\sm{\partial B}}
\lesssim\sum_{B\in\B_n,B\subset9B_F}\f{\sm{\partial B\cap\mb{\bigcup\B_n}}}{\sm{\partial B_F}},\]
which is uniformly bounded according to \Cref{cla_largeboundaryinball}.
Therefore we can conclude
\begin{align*}
\sm{\mb{\bigcup\B_n}}
&\lesssim\lambda^{-\f{d-1}d}\sum_{k\leq n}2^{k-n}\sum_{F\in\F_k}\sm{\partial F}.\\
\intertext{
So the interaction between the scales is small enough
so that we can just sum over all scales and obtain
}
\sm{\mb{\bigcup\B}}&\leq\sum_n\sm{\mb{\bigcup\B_n}}\\
&\lesssim\lambda^{-\f{d-1}d}\sum_k\sum_{n\geq k}2^{k-n}\sum_{F\in\F_k}\sm{\partial F}\\
&\lesssim\lambda^{-\f{d-1}d}\sum_k\sum_{F\in\F_k}\sm{\partial F}\\
&=\lambda^{-\f{d-1}d}\sum_{F\in\F}\sm{\partial F}.
\intertext{
Now we get back from \(\F\) to \(E\). 
Recall that for each \(F\in\F\) we have \(\lm{F\cap E}=\lm F/2\),
so that by \Cref{lem_isoperimetricconsequence} we have \(\sm{\partial F}\lesssim\sm{\mb E\cap F}\).
Because the balls in \(\F\) are disjoint, we can then conclude
}
\sm{\mb{\bigcup\B}}
&\lesssim\lambda^{-\f{d-1}d}\sum_{F\in\F}\sm{\mb E\cap F}\\
&\leq\lambda^{-\f{d-1}d}\sm{\mb E\cap\bigcup\B}.
\end{align*}

\end{proof}

%% file: small_average_local.tex
\begin{lem}\label{lem_volumemakesalmostdisjoint}
Let \(0\leq\lambda\leq2^{-\f{d+1}2}(d+1)^{-\f12}\) and let \(B,C\) be balls with \(\diam C\geq\diam B\) and
\(\lm{B\cap C}\leq\lambda\lm B\).
Then \((1-2(d+1)^{\f1{d+1}}\lambda^{\f2{d+1}})B\) and \(C\) are disjoint.
\end{lem}
For the application we only need that for \(\lambda\) small enough \(B\) and \((3/4)^{\f1d}C\) are disjoint.
Since \(\diam C\geq\diam B\) this follows if \((3/4)^{\f1d}B\) and \(C\) are disjoint.
The rate in \(\lambda\) plays no role.

\begin{proof}
After rescaling, rotation and translation it suffices to consider the case that there are
\(r\geq1\) and \(0<\varepsilon\leq2\) such that \(B=B(e_1,1)\) and \(C=B((\varepsilon-r)e_1,r)\).
\begin{figure}
\centering
\includegraphics{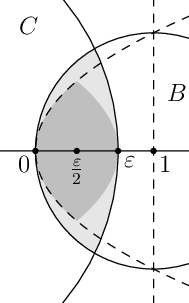}
\caption{The lower bound for \(\lm{B\cap C}\) in the proof of \Cref{lem_volumemakesalmostdisjoint}.}
\label{fig_intersectingballvolume}
\end{figure}
We bound \(\lm{B\cap C}\) from below by the marked area in \Cref{fig_intersectingballvolume}.
For \(x\in\mathbb{R}^d\) denote \(\bar x_1=(x_2,\ldots,x_d)\).
The two spheres \(\partial B\) and \(\partial C\) intersect in a plane orthogonal to \(e_1\)
that is between \(\f\varepsilon2e_1\) and \(\varepsilon e_1\).
Thus
\[
\Bigl\{x:\bar x_1^2< x_1<\f\varepsilon2\Bigr\}
\subset\Bigl\{x\in B:x_1<\f\varepsilon2\Bigr\}
\subset B\cap C
\]
and by symmetry and \(r\geq1\) also the image of the first set mirrored at \(x_1=\f\varepsilon2\) is contained in \(B\cap C\),
so that
\[
\lm{B\cap C}
>
2\lmb{\Bigl\{x:\bar x_1^2< x_1<\f\varepsilon2\Bigr\}}
=
2\int_0^{\f\varepsilon2}\sigma_{d-1}h^{\f{d-1}2}\intd h
=
2^{-\f{d-3}2}\f{\sigma_{d-1}}{d+1}\varepsilon^{\f{d+1}2}
.
\]
Therefore, since \(\lm{B\cap C}\leq\lambda\lm B=\lambda\sigma_d\)
we can conclude the following upper bound for \(\varepsilon\) using \Cref{eq_unitvolumeratio}.
\begin{align*}
\varepsilon^{\f{d+1}2}
&\leq
\f{\lambda(d+1)\sigma_d}{\sigma_{d-1}}2^{\f{d-3}2}
\leq
2^{\f{d-2}2}\f{d+1}{\sqrt d}\sqrt\pi\lambda
\leq
2^{\f{d+1}2}\f{d+1}{\sqrt{2d}}\lambda
\leq
2^{\f{d+1}2}(d+1)^{\f12}\lambda,\\
\varepsilon&\leq2(d+1)^{\f1{d+1}}\lambda^{\f2{d+1}}.
\end{align*}
This finishes the proof because \((1-\varepsilon)B\) and \(C\) are disjoint.
\end{proof}

\begin{lem}\label{lem_ballsreachinside}
Let \(0<\lambda\leq1\), let \(B\) be a ball and let \(\F\) be a set of balls with \(\lm{\bigcup\F\cap B}\geq\lambda\lm B\).
Then there is a ball \(F\in\F\) which intersects \((1-\lambda/d)B\).
\end{lem}

\begin{proof}
Since
\[
\lm{B\setminus(1-\lambda/d)B}
=
d\lm B\int_{1-\lambda/d}^1r^{d-1}\intd r
<
\lambda\lm B
,
\]
the union \(\bigcup\F\) cannot lie outside of \((1-\lambda/d)B\).
\end{proof}

\begin{proof}[Proof of \Cref{pro_levelsets_finite_l_local}]
According to \Cref{lem_boxing}, for every \(B\in\B\), almost every point in \(E\cap B\) is contained in a ball \(F\subset B\) with
\[\lm{F\cap E}=\f12\lm F.\]
Denote by \(\G\) the set of all such balls \(F\).
By scaling it suffices to consider the case that all balls in \(\G\) and \(\B\) have diameter at most \(1\).
We inductively build sequences \((\F_n)_{n=0}^{-\infty}\) and \((\G_n)_{n=0}^{-\infty}\) of subsets of \(\G\).
We denote \(\F_{>n}=\bigcup_{n<k\leq0}\F_n\) and \(\G_{>n}\) and \(\F_{n<\cdot\leq k}\) accordingly.
Assume we are at scale \(n\leq 0\).
Denote by \(\B_n\) the set of balls in \(\B\) with \(\diam B\in[\f12,1)2^n\).
Decompose \(\B_n\) into
\begin{align*}
\B_n^0&=\Bigl\{B\in\B_n:\lm{\bigcup5\F_{>n}\cap B}\leq\f\lambda2\lm{B}\Bigr\},\\
\B_n^1&=\Bigl\{B\in\B_n:\lm{\bigcup5\F_{>n}\cap B}>\f\lambda2\lm{B}\Bigr\}
\end{align*}
and decompose \(\B_n^1\) into
\begin{align*}
\B_n^{1,0}
&=
\Bigl\{B\in\B_n^1:
\lm{\bigcup\F_{>n}\cap B}
\leq
\f1{8^{\f{d+1}2}(d+1)^{\f{d+2}2}}
\lm{B}
\Bigr\}
,
\\
\B_n^{1,1}
&=
\Bigl\{B\in\B_n^1:
\lm{\bigcup\F_{>n}\cap B}
>
\f1{8^{\f{d+1}2}(d+1)^{\f{d+2}2}}
\lm{B}
\Bigr\}
.
\end{align*}
Denote by \(\G_n\) the set of balls \(G\in\G\) with \(\diam G\in[\f12,1)2^n\) which intersect
\(E\setminus\bigcup5\F_{>n}\)
or are for some \(k\geq n\) and some \(B\in\B_k^{1,0}\) contained in
\(B\setminus\bigcup\F_{n<\cdot\leq k}\).
Set \(\F_n\) to be a maximal disjoint subcollection of \(\G_n\).

Denote \(\F=\bigcup_n\F_n\), \(\B_0=\bigcup_n\B_n^0\) and \(\B^{1,0}\) and \(\B^{1,1}\) accordingly.
Here are a few properties of those ball collections.

\begin{enumerate}
\item
\label{lem_FncoversGn}
The collection \(5\F_n\) is a cover of \(\bigcup\G_n\).
\item
\label{lem_FcoversE}
The collection \(5\F\) covers almost all of \(E\).
\item
\label{lem_Bdecomposition}
The balls in \((3/4)^{\f1d}\F\) are disjoint.
\item
\label{lem_smallFcoverB0}
If \(B\in\B_n^0\) then \(5\F_{\leq n}\) covers at least \(\f\lambda2\) of \(B\).
\item
\label{lem_smallFcoverB10}
If \(B\in\B_n^{1,0}\) then \(5\F_{\leq n}\) covers at least \(\lambda\) of \(B\).
\end{enumerate}

\begin{proof}
\begin{enumerate}
\item
By the maximality of \(\F_n\) every \(G\in\G_n\) intersects an \(F\in\F_n\).
Since \(\diam G\leq2\diam F\) this means \(G\subset5F\).
\item
Let \(G\in\G\) be a ball and \(n\in\mathbb{Z}\) with \(\diam G\in[\f12,1)2^n\). 
Then \(G\) intersects \(E\), so that by definition of \(\G_n\) we have \(G\cap E\subset\bigcup5\F_{>n}\) or \(G\in\G_n\).
By \Cref{lem_FncoversGn} we can conclude \(G\cap E\subset\bigcup5\F_{\geq n}\) in either case.
Since \(\G\) covers almost all of \(E\) this means so does \(5\F\).
\item
For each \(n\) the balls in \(\F_n\) are disjoint.
It remains to show that they are disjoint from the balls in \((3/4)^{\f1d}\F_{>n}\). 
So assume \(F\in\F_n\).
If \(F\) was chosen because it intersects \(E\setminus\bigcup5\F_{>n}\)
then it doesn't intersect \(\F_{>n}\).
It remains to consider the case that there is a \(k\geq n\) and a \(B\in\B_k^{1,0}\) such that \(F\subset B\)
and \(F\) does not intersect any \(G\in\F_{n<\cdot\leq k}\).
Since \(B\in\B_k^{1,0}\), for every \(G\in\F_{>k}\) we have
\[
\lm{B\cap G}
\leq
\f1{8^{\f{d+1}2}(d+1)^{\f{d+2}2}}
\lm B,
\]
so that by \Cref{lem_volumemakesalmostdisjoint}
the balls \(\bigl(1-\f1{4(d+1)}\bigr)B\) and \(G\) are disjoint.
Since \(\bigl(\f34\bigr)^{\f1d}\leq1-\f1{4(d+1)}\) and \(\diam G\geq\diam B\)
this means that \((3/4)^{\f1d}G\) and \(B\) are disjoint, too.
Hence also \(F\) and \(\bigcup(3/4)^{\f1d}\F_{>k}\) are disjoint.
\item
For every \(B\in\B_n^0\) we have \(\lm{B\cap E}\geq\lambda\lm B\).
Thus since \(5\F\) covers almost all of \(E\) and
\[\lm{\bigcup5\F_{>n}\cap B}\leq\f\lambda2\lm{B}\]
we must have
\[\lm{\bigcup5\F_{\leq n}\cap B}\geq\f\lambda2\lm{B}.\]
\item
Let \(B\in\B_n^{1,0}\).
It suffices to show that \(5\F_{\leq n}\) covers \(E\cap B\).
By the construction of \(\G\) using \Cref{lem_boxing} almost all of \(B\cap E\) is covered by the union of all \(G\in\G\) with \(G\subset B\) and \(\diam G<2^n\).
Thus it suffices to show for each such \(G\) that \(G\cap E\) is contained in \(5\F_{\leq n}\).
Take \(k\leq n\) with \(\diam G\in[\f12,1)2^k\).
If \(G\cap E\) is not contained in \(\bigcup5\F_{k<\cdot\leq n}\) then \(G\in\G_k\) and thus by \Cref{lem_FncoversGn} we have \(G\subset\bigcup5\F_k\).
\end{enumerate}
\end{proof}

Denote \(\tilde\B=\B^0\cup\B^{1,0}\) so that \(\B=\tilde\B\cup\B^{1,1}\).
Then by \Cref{lem_boundaryofunion} we have 
\[\mb{\bigcup\B}\subset\mb{\bigcup\tilde\B}\cup\bigl(\mb{\bigcup\B^{1,1}}\setminus\mc{\bigcup\tilde\B}\bigr).\]
Note that for a finite union of balls the topological and measure theoretical notions agree up to a set of \(d-1\) dimensional measure zero.
By \Cref{lem_ballsreachinside} for every \(B\in\B^{1,1}\) there is an \(F\in\F\) with \(\diam F>\diam B\)
that intersects \((1-8^{-\f{d+1}2}d^{-\f{d+4}2})B\).
Because \(F\) came about using \Cref{lem_boxing}, it is further contained in a ball \(B_F\in\B\).
Since \(\diam B<\diam B_F\) we have \(B\neq B_F\).
For each \(F\in\F\) denote by \(\B(F)\) the set of \(B\in\B\) with \(\diam B<\diam F\)
such that \(F\) intersects \((1-8^{-\f{d+1}2}d^{-\f{d+4}2})B\).
Then
\begin{align*}
\partial\bigcup\B^{1,1}\setminus\cl{\bigcup\tilde\B}
&\subset\partial\bigcup\B^{1,1}\setminus\bigcup\B\\
&\subset\bigcup_{F\in\F}\partial\bigcup(\B^{1,1}\cap\B(F))\setminus\bigcup\B\\
&\subset\bigcup_{F\in\F}\partial\bigcup(\B^{1,1}\cap\B(F))\setminus(\bigcup\B(F)\cup B_F)\\
&\subset\bigcup_{F\in\F}\partial\bigcup\B(F)\setminus(\bigcup\B(F)\cup B_F)\\
&=\bigcup_{F\in\F}\partial\bigcup\B(F)\setminus B_F\\
&\subset\bigcup_{F\in\F}\partial(F\cup\bigcup\B(F)).
\end{align*}
Thus \Cref{pro_levelsets_finite_g} implies
\[
\sm{\partial\bigcup\B^{1,1}\setminus\cl{\bigcup\tilde\B}}
\lesssim\sum_{F\in\F}\sm{\partial F}
\]
Recall that we made \((3/4)^{\f1d}\F\) disjoint
and that by \Cref{lem_boxing} for each \(F\in\F\) we have \(F\subset\bigcup\B\) 
and \(\lm{F\cap E}=\lm F/2\).
Thus \(\lm{(\f34)^{\f1d}F\cap E}\in[\f14,\f34]\lm F\)
and so by \Cref{lem_isoperimetricconsequence} we can conclude
\begin{align}
\nonumber
\sum_{F\in\F}\sm{\partial F}
&=(4/3)^{\f{d-1}d}
\sum_{F\in\F}\sm{\partial(3/4)^{\f1d}F}
\\
\nonumber
&\lesssim
\sum_{F\in\F}\sm{\mb E\cap(3/4)^{\f1d}F}
\\
\label{eq_boundsumF}
&\leq
\sm{\mb E\cap\bigcup\B}
.
\end{align}

It remains to prove
\begin{equation}\label{eq_boundtildeB}
\sm{\partial\bigcup\tilde\B}\lesssim\lambda^{-\f{d-1}d}\sum_{F\in\F}\sm{\partial F}.
\end{equation}
For \(n\in\mathbb{Z}\) denote by \(\tilde\B_n\) the set of balls \(B\in\tilde\B\) with \(\diam B\in[\f12,1)2^n\).
Let \(B\in\tilde\B_n\) and let \(F\in\F_{\leq n}\) be a ball such that \(5F\) intersects \(B\).
Then \(F\subset21B\).
By \Cref{lem_smallFcoverB0,lem_smallFcoverB10} this means
\begin{equation}\label{eq_atleastlambda2cover}
\f\lambda2\lm B\leq\lm{B\cap\bigcup5\F_{\leq n}}\leq\sum_{F\in\F_{\leq n},F\subset21B}\lm{5F\cap B}.
\end{equation}
For each \(k\in\mathbb{Z}\) denote by \(\tilde\F_k\) the set of balls \(F\in\F\) with \(\diam F\in[\f12,1)2^k\lambda^{\f1d}\).
We make a case distinction.
If there is a \(k\geq n\) and a ball \(F\in\tilde\F_k\) with \(F\subset21B\) we have
\begin{align}
\nonumber\sm{\partial B}&=\f{\sm{\partial B}}{\sm{\partial F}}\sm{\partial F}\\
\nonumber&\leq2^{2(d-1)}\f{2^{n(d-1)}}{\lambda^{\f{d-1}d}2^{k(d-1)}}\sm{\partial F}\\
\label{eq_caseFbig}&\sim2^{(n-k)(d-1)}\lambda^{-\f{d-1}d}\sm{\partial F},
\end{align}
and we are done with this case for the moment.
Now assume all balls \(F\in\F\) with \(F\subset21B\) are contained in \(\tilde\F_{<n}\).
Then by \cref{eq_atleastlambda2cover} we have
\begin{align}
\nonumber\sm{\partial B}
&\leq2
\sum_{k<n}\sum_{F\in\tilde\F_k,F\subset21B}
\f{\lm{5F\cap B}}{\lambda\lm{B}}\sm{\partial B}
\\
\nonumber
&\lesssim
\sum_{k<n}\sum_{F\in\tilde\F_k,F\subset21B}
\Bigl(\f{\lm F}{\lambda\lm{B}}\Bigr)^{\f1d}
\lambda^{-\f{d-1}d}
\Bigl(\f{\lm F}{\lm{B}}\Bigr)^{\f{d-1}d}
\sm{\partial B}
\\
\label{eq_dbtodf2}
&\sim
\sum_{k<n}\sum_{F\in\tilde\F_k,F\subset21B}2^{k-n}\lambda^{-\f{d-1}d}\sm{\partial F}.
\end{align}

If \(d=1\) then \Cref{pro_levelsets_finite_l_local} is straightforward to prove directly,
so it suffices to consider \(d\geq2\).
There we can combine \cref{eq_caseFbig,eq_dbtodf2} into
\[\sm{\partial B}\lesssim\lambda^{-\f{d-1}d}\sum_k2^{-|k-n|}\sum_{F\in\tilde\F_k,F\subset21B}\sm{\partial F}\]
for simplicity.
This estimate can be seen as a way to distribute \(\sm{\partial B}\) over the balls \(F\) that it contains. 
The next step will be to turn the dependence around, and see for a fixed ball \(F\in\F\) for how much of \(\sm{\mb{\bigcup\tilde\B}}\) it is responsible.
Since \(\tilde\B_n\) is finite we have
\begin{align*}
\sm{\mb{\bigcup\tilde\B_n}}&=\sum_{B\in\tilde\B_n}\sm{\partial B\cap\mb{\bigcup\tilde\B_n}}\\
\intertext{and we multiply each summand by a number bounded from below according to \cref{eq_dbtodf2}}
\sm{\mb{\bigcup\tilde\B_n}}
&\lesssim\sum_{B\in\tilde\B_n}\f{\sm{\partial B\cap\mb{\bigcup\tilde\B_n}}}{\sm{\partial B}}\sum_k\sum_{F\in\tilde\F_k,F\subset21B}2^{-|k-n|}\lambda^{-\f{d-1}d}\sm{\partial F}\\
&=\lambda^{-\f{d-1}d}\sum_k2^{-|k-n|}\sum_{F\in\tilde\F_k}\sm{\partial F}\sum_{B\in\tilde\B_n,21B\supset F}\f{\sm{\partial B\cap\mb{\bigcup\tilde\B_n}}}{\sm{\partial B}}.
\end{align*}
We have reorganized \(\mb{\bigcup\tilde\B_n}\) according to the balls in \(\F\). 
We want to bound the contribution of each ball \(F\in\F\) uniformly.
For each \(F\in\tilde\F_k\) for which there is a ball \(B\in\tilde\B_n\) with \(F\subset21B\), denote by \(B_F\) a largest such \(B\).
Then for all \(B\in\tilde\B_n\) with \(F\subset21B\) have \(B\subset3B_F\). 
Thus,
\[
\sum_{B\in\tilde\B_n,21B\supset F}\f{\sm{\partial B\cap\mb{\bigcup\tilde\B_n}}}{\sm{\partial B}}
\lesssim\sum_{B\in\tilde\B_n,B\subset63B_F}\f{\sm{\partial B\cap\mb{\bigcup\tilde\B_n}}}{\sm{\partial B_F}}
\]
which is uniformly bounded according to \Cref{cla_largeboundaryinball}.
Therefore we can conclude
\begin{align*}
\sm{\mb{\bigcup\tilde\B_n}}
&\lesssim\lambda^{-\f{d-1}d}\sum_k2^{-|k-n|}\sum_{F\in\tilde\F_k}\sm{\partial F}.\\
\intertext{
So the interaction between the scales is small enough
that we can just sum over all scales and obtain
}
\sm{\mb{\bigcup\tilde\B}}&\leq\sum_n\sm{\mb{\bigcup\tilde\B_n}}\\
&\lesssim\lambda^{-\f{d-1}d}\sum_k\sum_n2^{-|k-n|}\sum_{F\in\tilde\F_k}\sm{\partial F}\\
&\lesssim\lambda^{-\f{d-1}d}\sum_k\sum_{F\in\tilde\F_k}\sm{\partial F}\\
&=\lambda^{-\f{d-1}d}\sum_{F\in\F}\sm{\partial F},
\end{align*}
and we have proven \cref{eq_boundtildeB} which was all that remained to finish the proof of \Cref{pro_levelsets_finite_l_local}.
\end{proof}